\documentclass[reqno,11pt]{amsart}
\usepackage[foot]{amsaddr}
\usepackage{bbold}
\usepackage{charter}
\usepackage[margin=1in]{geometry}
\usepackage[colorlinks=true,linktoc=all,linkcolor=blue,citecolor=blue]{hyperref}

\theoremstyle{plain}
\newtheorem{theorem}{Theorem}[section]

\newtheorem{lemma}[theorem]{Lemma}
\newtheorem{corollary}[theorem]{Corollary}
\theoremstyle{definition}
\newtheorem{definition}[theorem]{Definition}

\newtheorem{example}[theorem]{Example}

\newcommand{\D}{\mathbb{D}}
\newcommand{\N}{\mathbb{N}}
\newcommand{\bigchi}{\mbox{\Large$\chi$}}
\newcommand{\Linf}{L^{\infty}}
\newcommand{\Lmuinf}{L_{\hspace{-.2ex}\mu}^{\infty}}

\author{Robert F.~Allen\textsuperscript{1} and Matthew A.~Pons\textsuperscript{2}}
\address{\textsuperscript{1}Department of Mathematics and Statistics,University of Wisconsin-La Crosse,}
\address{\textsuperscript{2}Department of Mathematics,North Central College}
\email{rallen@uwlax.edu, mapons@noctrl.edu}


\date{}
\title[Composition Operators on $\Lmuinf$]{Composition Operators on Weighted Banach Spaces of a Tree}

\begin{document}

\begin{abstract}
We study composition operators on the weighted Banach spaces of an infinite tree.  We characterize the bounded and the compact operators, as well as determine the operator norm and the essential norm.  In addition, we study the isometric composition operators.  
\end{abstract}

\maketitle

\section{Introduction}
Let $X$ be a Banach space of complex-valued functions defined on a set $\Omega$ and let $\varphi$ be a self-map of $\Omega$.  The composition operator $C_\varphi$ with symbol $\varphi$ is the linear operator defined by $$C_\varphi f = f\circ\varphi$$ for all $f \in X$.  This operator is well studied when $X$ is a Banach (or Hilbert) space of analytic functions.  The interested reader is directed to \cite{CowenMacCluer:1995} for an excellent treatise on the subject.

In recent years, the study of operators on discrete structures has been introduced to the literature.  A particular structure of interest is an infinite tree, which can be viewed as a complete metric space with respect to the edge-counting metric.  In several instances, spaces of functions on these trees can be viewed as discrete analogs to classical function spaces, such as the Bloch space and the space of bounded analytic functions on the open unit disk $\D$.

The first of these discrete spaces of interest is the Lipschitz space, as defined in \cite{ColonnaEasley:2010}.  As defined, this space is the natural discrete analog of the Bloch space.  The multiplication and composition operators were studied in \cite{ColonnaEasley:2010} and \cite{AllenColonnaEasley:2014}, respectively.  From the study of the multiplication operators, an entire family of spaces, called iterated Logarithmic Lipschitz spaces, were created \cite{AllenColonnaEasley:2012}.  These spaces are discrete analogs to the logarithmic Bloch spaces.  Even more recently, multiplication and composition operators on discrete analogies of the Hardy spaces $H^p$ have been studied in \cite{MuthukumarPonnusamy:2016} and \cite{MuthukumarPonnusamy-I:2016}.

A very natural space of functions on trees is the space of bounded functions, denoted $\Linf$.  This is, of course, the discrete analog to the space of bounded analytic functions $H^\infty$.  In \cite{ColonnaEasley:2012}, the multiplication operators between the Lipschitz space and $\Linf$ were studied.  

A natural extension of $H^\infty$ is the so-called weighted Banach spaces.  For a positive, continuous function $\mu$ on $\Omega$, the weighted Banach space $H_{\hspace{-.2ex}\mu}^{\infty}$ is the space of functions $f$ on $\Omega$ for which the sup-norm of $\mu |f|$ is finite.  These spaces have been studied extensively, and the interested reader is directed to \cite{BonetDomanskiLindstromTaskinen:1998}, \cite{BonetDomanskiLindstrom:1999}, \cite{ContrerasHernandez-Diaz:2000}, and \cite{BonetLindstromWolf:2008} for more information.

In \cite{AllenCraig:2015}, the weighted Banach spaces on a tree were defined, and the multiplication operators studied.  In this paper, we wish to further study these weighted Banach spaces, as well as study the composition operators on them.  The results of this work are compatible with the results found in \cite{BonetDomanskiLindstromTaskinen:1998}, \cite{BonetDomanskiLindstrom:1999}, \cite{ContrerasHernandez-Diaz:2000}, and \cite{BonetLindstromWolf:2008}, and as corollaries our results apply to composition operators on $\Linf$.

\subsection{Organization of this Paper}
In Section \ref{Section:Spaces}, we study the weighted Banach spaces in more detail than in \cite{AllenCraig:2015}.  In Section \ref{Section:Boundedness}, we characterize the symbols which induce bounded composition operators on the weighted Banach spaces, as well as determine the operator norm.  

In Section \ref{Section:Compactness}, we characterize the symbols which induce compact composition operators on the weighted Banach spaces, as well as determine the essential norm.  Finally, in Section \ref{Section:Isometries}, we provide various necessary conditions and sufficient conditions for the composition operator to be an isometry.  For the case of the space $\Linf$, we characterize the symbols which induce isometric composition operators.

\section{Weighted Banach Spaces of a Tree}\label{Section:Spaces}
A tree $T$ is a locally finite, connected, and simply-connected graph, which, as a set, we identify with the collection of its vertices. Two vertices $v$ and $w$ are called neighbors if there is an edge $[v,w]$ connecting them, and we adopt the notation $v\sim w$. A vertex is called terminal if it has a unique neighbor. A path is a finite or infinite sequence of vertices $[v_0,v_1,\dots]$ such that $v_k\sim v_{k+1}$ and  $v_{k-1}\neq v_{k+1}$, for all $k$. 

In this article, we shall assume the tree $T$ to be without terminal vertices (and hence infinite), and rooted at a vertex $o$.  The length of a finite path $[v=v_0,v_1,\dots,w=v_n]$ (with $v_k\sim v_{k+1}$ for $k=0,\dots, n$) is the number $n$ of edges connecting $v$ to $w$. The distance, $d(v,w)$, between vertices $v$ and $w$ is the length of the unique path connecting $v$ to $w$. The distance between the root $o$ and a vertex $v$ is called the length of $v$ and is denoted by $|v|$. 

A function on a tree is a complex-valued function on the set of its vertices. 

\begin{definition} The space of bounded functions on a tree $T$, denoted $\Linf$, is the set of functions $f$ on $T$ such that $$\displaystyle\sup_{v \in T} |f(v)| < \infty.$$ 
\end{definition}

As with the case of $H^\infty$, $\Linf$ is a Banach space under the norm $\|f\|_\infty = \displaystyle\sup_{v \in T} |f(v)|$.  To generalize the weighted Banach spaces of a tree, we take as weights any positive function on $T$.

\begin{definition} Given a positive function $\mu$ on a tree $T$, the weighted Banach space on $T$, denoted $\Lmuinf$, is defined as the set of functions $f$ on $T$ for which $$\sup_{v \in T} \mu(v)|f(v)| < \infty.$$
\end{definition} In \cite{AllenCraig:2015}, it was shown that $\Lmuinf$ is a functional Banach space (see Definition 1.1 of \cite{CowenMacCluer:1995}) under the norm $\|f\|_\mu = \displaystyle\sup_{v \in T} \mu(v)|f(v)|.$  

Clearly for the constant weight $\mu \equiv 1$, $\Lmuinf$ is $\Linf$.  In fact, if the weight is bounded above and away from zero, then the functions in $\Lmuinf$ are precisely the bounded functions on $T$ and the norms $\|\cdot\|_\mu$ and $\|\cdot\|_\infty$ are equivalent.  

\begin{theorem} Suppose there exist positive constants $m$ and $M$ such that $m \leq \mu(v) \leq M$ for all $v \in T$.  Then $\Lmuinf = \Linf$ with equivalent norms.
\end{theorem}


The next two results show that normalized characteristic functions (and in fact characteristic functions) belong to the $\Lmuinf$ spaces.   

\begin{lemma}\label{Lemma:chi function} Let $w$ be a fixed vertex in $T$.  Then the function $f(v) = \displaystyle\frac{1}{\mu(v)}\chi_{w}(v)$ is an element of $\Lmuinf$ with $\|f\|_\mu = 1$.  
\end{lemma}

\noindent Sequences of these normalized characteristic functions are used to characterize compact composition operators in Section \ref{Section:Compactness}.
\begin{lemma}\label{lemma:bounded family} Let $(w_n)$ be a sequence of vertices with $|w_n| \to \infty$ as $n \to \infty$.  Then the sequence of functions $(f_n)$ defined by $f_n(v) = \frac{1}{\mu(v)}\bigchi_{w_n}(v)$ is a bounded sequence in $\Lmuinf$ converging to 0 pointwise.
\end{lemma}

\begin{proof}
	By Lemma \ref{Lemma:chi function} for each $n \in \N$, $f_n$ is an element of $\Lmuinf$ with $\|f_n\|_\mu = 1$.  Thus $(f_n)$ is a bounded sequence in $\Lmuinf$.  Fix a vertex $w \in T$.  We will show that $(f_n(w))$ converges to 0.  If $w \neq w_n$ for all $n \in \N$, then $f_n(w) = 0$ for all $n \in \N$.  Suppose $w = w_k$ for some $k \in \N$.  Then for all $n \geq k+1$, we have $f_n(w) = 0$.  Thus $(f_n)$ converges pointwise to 0 on $T$.
\end{proof}

The following three results prove point evaluation is a bounded linear functional on $\Lmuinf$, as well as how the point evaluation functionals compare.

\begin{lemma}\label{Lemma:Point-evaluation bound}
Let $f \in \Lmuinf$.  Then $|f(v)| \leq \displaystyle\frac{1}{\mu(v)}\|f\|_\mu$ for all $v \in T$.
\end{lemma}


\begin{lemma} For $v \in T$, the norm of the point evaluation functional $K_v$ on $\Lmuinf$ is given by $\|K_v\| = \displaystyle\frac{1}{\mu(v)}$.
\end{lemma}

\begin{proof}
	Let $f \in \Lmuinf$ such that $\|f\|_\mu = 1$.  Then by Lemma \ref{Lemma:Point-evaluation bound}, $$\|K_v\| = \sup_{\|f\|_\mu = 1} |K_vf| = \sup_{\|f\|_\mu = 1} |f(v)| \leq \sup_{\|f\|_\mu = 1} \frac{1}{\mu(v)} \|f\|_\mu  = \frac{1}{\mu(v)}$$ for $v \in T$.  By Lemma \ref{Lemma:chi function}, the function $g(w) = \displaystyle\frac{1}{\mu(v)}\bigchi_v(w)$ is an element of $\Lmuinf$ with $\|g\|_\mu = 1$.  Observe, $$\frac{1}{\mu(v)} = |g(v)| = |K_v g| \leq \sup_{\|f\|_\mu = 1} |K_v f| = \|K_v\|.$$  Thus $\|K_v\| = \displaystyle\frac{1}{\mu(v)}$ as desired.
\end{proof}

\begin{lemma}\label{Lemma:equal functionals} If $K_v = K_w$, for $v,w \in T$, then $v=w$.
\end{lemma}

\begin{proof}
	Suppose $K_v = K_w$, but $v \neq w$.  Then $K_v f = K_w f$ for all $f \in \Lmuinf$.  In particular, for $f = \bigchi_v$, we observe $K_v f = \bigchi_v(v) = 1$ but $K_w f = \bigchi_v(w) = 0$ since $v \neq w$.  Thus $K_v \neq K_w$, which is a contradiction.
\end{proof}

\section{Boundedness and Operator Norm}\label{Section:Boundedness}
In this section we characterize the bounded composition operators on $\Lmuinf$.  In addition, we determine the operator norm of $C_\varphi$.  The interaction between the weight $\mu$ and the symbol $\varphi$ for bounded composition operators is characterized in the following quantity.  For a function $\varphi:T \to T$, define $$\sigma_\varphi = \displaystyle\sup_{v \in T} \frac{\mu(v)}{\mu(\varphi(v))}.$$

\begin{theorem}\label{theorem:boundedness}
Let $\varphi$ be a self-map of a tree $T$.  Then $C_\varphi$ is bounded on $\Lmuinf$ if and only if $\sigma_\varphi$ is finite.  Furthermore $\|C_\varphi\| = \sigma_\varphi$.
\end{theorem}

\begin{proof}
Suppose $C_\varphi$ is bounded on $\Lmuinf$.  Define the function $g(v) = \displaystyle\frac{1}{\mu(v)}$.  Clearly $g \in \Lmuinf$ with $\|g\|_\mu = 1$.  Observe \begin{equation}\label{Equation:C_varphi-lower-bound} \|C_\varphi g\|_\mu = \sup_{v \in T} \mu(v)|g(\varphi(v))| = \sup_{v \in T} \displaystyle\frac{\mu(v)}{\mu(\varphi(v))} = \sigma_\varphi.\end{equation}  Since $C_\varphi$ is bounded on $\Lmuinf$, we have $\sigma_\varphi < \infty$.

Now suppose $\sigma_\varphi$ is finite and let $f \in \Lmuinf$.  From Lemma \ref{Lemma:Point-evaluation bound} we obtain 
\begin{equation}\label{Equation:C_varphi-upper-bound} \|C_\varphi f\|_\mu = \sup_{v \in T} \mu(v)|f(\varphi(v))| \leq \sup_{v \in T} \displaystyle\frac{\mu(v)}{\mu(\varphi(v))}\|f\|_\mu = \sigma_\varphi\|f\|_\mu.\end{equation}  Thus $C_\varphi$ is bounded as an operator on $\Lmuinf$.  Finally, we see that \eqref{Equation:C_varphi-lower-bound} implies that $\sigma_\varphi \leq \|C_\varphi\|$, where as \eqref{Equation:C_varphi-upper-bound} implies that $\|C_\varphi\| \leq \sigma_\varphi$.  Therefore, $\|C_\varphi\| = \sigma_\varphi$.
\end{proof}

\begin{corollary}
Every self-map of a tree $T$ induces a bounded composition operator $C_\varphi$ on $\Linf$.  Moreover, $\|C_\varphi\| = 1$.
\end{corollary}

We next consider which weighted Banach spaces share this property with $\Linf$, that is for what $\mu$ is it true that every self-map $\varphi$ induces a bounded composition operator on $\Lmuinf$.  At first glance, one might consider bounded weights to work.  As the next example illustrates, this is not sufficient. 

\begin{example}
Suppose $\mu(v) = \displaystyle\frac{1}{|v|}$ for $v \in T^*$ and $\mu(o) = 2$.  Then $0 < \mu(v) \leq 2$ for all $v \in T$.  Define $\varphi:T\to T$ to map a vertex $v$ to a vertex of length $2^{|v|}$ for all $v \in T$.  Let $(v_n)$ be a sequence of vertices such that $|v_n| \to \infty$ as $n \to \infty$.  Then $$\frac{\mu(v_n)}{\mu(\varphi(v_n))} = \frac{|\varphi(v_n)|}{|v_n|} = \frac{2^{|v_n|}}{|v_n|} \to \infty$$ as $|v_n| \to \infty$.  Thus by Theorem \ref{theorem:boundedness}, $C_\varphi$ is not bounded.
\end{example}

\begin{theorem}
Every self-map of a tree $T$ induces a bounded composition operator on $\Lmuinf$ if and only if there exist positive constants $m$ and $M$ such that $m \leq \mu(v) \leq M$ for all $v \in T$.
\end{theorem}

\begin{proof}
Suppose there exist positive constants $m$ and $M$ such that $m \leq \mu(v) \leq M$ for all $v \in T$.  Let $\varphi$ be any self-map of $T$.  Then $$\sigma_\varphi = \sup_{v \in T} \frac{\mu(v)}{\mu(\varphi(v))} \leq \frac{M}{m} < \infty.$$  Thus $C_\varphi$ is bounded on $\Lmuinf$.
	
Now suppose every self-map of $T$ induces a bounded composition operator on $\Lmuinf$.  First assume $\mu$ is unbounded, i.e. there exists a sequence $(v_n)$ of vertices such that $\mu(v_n) \to \infty$ as $n \to \infty$.  By Theorem \ref{theorem:boundedness}, the composition operator induced by the constant map $\varphi(v) = o$ is not bounded on $\Lmuinf$ since 
$$\frac{\mu(v_n)}{\mu(\varphi(v_n))} = \frac{\mu(v_n)}{\mu(o)} \to \infty$$ as $n \to \infty$, which is a contradiction.  Thus there must exist a positive constant $M$ such that $\mu(v) \leq M$ for all $v \in T$.
	
Finally, assume $\mu$ is not bounded away from zero, i.e. there exists a sequence $(v_n)$ of vertices such that $\mu(v_n) \to 0$ as $n \to \infty$.  We may construct a subsequence of $(v_n)$, which we denote $(w_n)$ with the property that for all $n \in \N$, $\mu(w_n) < \mu(v_n)^2$.  Define $\varphi(v_n) = w_n$ and $o$ otherwise. By Theorem \ref{theorem:boundedness}, the composition operator induced by $\varphi$ is not bounded on $\Lmuinf$ since 
$$\frac{\mu(v_n)}{\mu(\varphi(v_n))} > \frac{1}{\mu(v_n)} \to \infty$$ as $n \to \infty$, which is a contradiction.  Thus there must exist a positive constant $m$ such that $\mu(v) \geq m$ for all $v \in T$.
\end{proof}

\section{Compactness and Essential Norm}\label{Section:Compactness}
In this section, we characterize the compact composition operators on $\Lmuinf$, as well as determine the essential norm.  The means by which we will characterize the compact composition operators is Lemma \ref{Lemma:Compactness-criterion}.  

\begin{lemma}\cite[Lemma 2.5]{AllenCraig:2015}\label{Lemma:Compactness-criterion} 
	Let $X,Y$ be two Banach spaces of functions on a tree $T$.  Suppose that
	\begin{enumerate}
		\item[\normalfont{(i)}] the point evaluation functionals of $X$ are bounded,
		\item[\normalfont{(ii)}] the closed unit ball of $X$ is a compact subset of $X$ in the topology of uniform convergence on compact sets,
		\item[\normalfont{(iii)}] $A:X \to Y$ is bounded when $X$ and $Y$ are given the topology of uniform convergence on compact sets.
	\end{enumerate}  Then $A$ is a compact operator if and only if given a bounded sequence $(f_n)$ in $X$ such that $f_n \to 0$ pointwise, then the sequence $(A f_n)$ converges to zero in the norm of $Y$.
\end{lemma}

\noindent In \cite{AllenCraig:2015}, it was determined that $\Lmuinf$ satisfies the conditions of Lemma \ref{Lemma:Compactness-criterion}.  

For a bounded operator $A$ on a Banach space $X$, the essential norm of $A$ is defined to be the distance from $A$ (in the operator norm) to the ideal of compact operators, or $$\|A\|_e=\inf\{\|A-K\|_X:K\textup{ is compact}\}.$$  It follows that $A$ is compact if and only if $\|A\|_e=0$.  To determine the essential norm of a composition operator, we will employ the following sequence of compact operators. First, for $f\in\Lmuinf$ and $n\in\N$, define a function $f_n\in\Lmuinf$ by $$f_n(v)= \begin{cases}f(v) & \text{if } |v| \leq n;\\0 & \text{if } |v| > n.\end{cases}$$  Then define the operator $A_n$ by $A_n f=f_n$.  It is easy to see that the operator is linear. The following lemma captures the other most relevant properties of these operators and we omit the proof.

\begin{lemma}\label{Lemma:compactsequence}
For each $n\in\N$, the operator $A_n$ is compact on $\Lmuinf$ with $\|A_n\|\leq 1$ and $\|I-A_n\|\leq 1.$
\end{lemma}

\begin{theorem}\label{Theorem:compactness} Let $\varphi$ be a self-map of $T$ such that $C_\varphi$ is bounded on $\Lmuinf$. If $\varphi$ has infinite range, then 
\begin{equation}\label{Equation:essentialnorm}\|C_\varphi\|_e=\lim_{N \to \infty} \sup_{|\varphi(v)| \geq N} \frac{\mu(v)}{\mu(\varphi(v))}.\end{equation} Furthermore, $C_\varphi$ is compact if and only if $\varphi$ has finite range or \begin{equation}\label{limit condition}
	\lim_{N \to \infty} \sup_{|\varphi(v)| \geq N} \frac{\mu(v)}{\mu(\varphi(v))}= 0.
\end{equation}
\end{theorem}

\begin{proof}
To determine the essential norm of $C_\varphi$, we assume $\varphi$ has infinite range so that the limit in equation \eqref{Equation:essentialnorm} is defined.  Observe $C_\varphi A_n$ is compact for all $n \in \N$ since $C_\varphi$ is bounded and $A_n$ is compact from Lemma \ref{Lemma:compactsequence}.  From the definition of the essential norm, we have \begin{equation}
	\label{Inequality:essentialnormupper}\|C_\varphi\|_e\leq \|C_\varphi-C_\varphi A_n\|=\sup_{\|f\|_\mu\leq1}\sup_{v\in T}\mu(v)\left|(C_\varphi(I-A_n)f)(v)\right|\end{equation} for every $n\in\N$.  Now fix $N\in\N$. We define $$R_N(n)=\sup_{\|f\|_\mu\leq1}\sup_{|\varphi(v)|\geq N}\mu(v)\left|(C_\varphi(I-A_n)f)(v)\right|$$ and $$S_N(n)= \sup_{\|f\|_\mu\leq1}\sup_{|\varphi(v)|\leq N}\mu(v)\left|(C_\varphi(I-A_n)f)(v)\right|.$$  Then, from inequality \eqref{Inequality:essentialnormupper} we obtain $$\|C_\varphi\|_e\leq \max\{R_N(n),S_N(n)\}$$ for each $n, N\in \N$.  We now consider the case $n >N$. From Lemma \ref{Lemma:compactsequence}, we obtain  $$\begin{aligned}R_N(n)&=\sup_{\|f\|_\mu\leq1}\sup_{|\varphi(v)|\geq N}\frac{\mu(v)}{\mu(\varphi(v))}\mu(\varphi(v))\left|(C_\varphi(I-A_n)f)(v)\right| \\
	&=\sup_{\|f\|_\mu\leq1}\sup_{|\varphi(v)|\geq N}\frac{\mu(v)}{\mu(\varphi(v))}\mu(\varphi(v))\left|((I-A_n)f)(\varphi(v))\right|\\
	&\leq\sup_{\|f\|_\mu\leq1}\sup_{|\varphi(v)|\geq N}\frac{\mu(v)}{\mu(\varphi(v))}\sup_{w\in T}\mu(w)\left|((I-A_n)f)(w)\right|\\
	&=\sup_{|\varphi(v)|\geq N}\frac{\mu(v)}{\mu(\varphi(v))}\sup_{\|f\|_\mu\leq1}\sup_{w\in T}\mu(w)\left|((I-A_n)f)(w)\right|\\
	&=\sup_{|\varphi(v)|\geq N}\frac{\mu(v)}{\mu(\varphi(v))}\|I-A_n\|\\
	&\leq \sup_{|\varphi(v)|\geq N}\frac{\mu(v)}{\mu(\varphi(v))}.
	\end{aligned}$$ 
	
	\noindent Next, observe that $$\begin{aligned}S_N(n)&=\sup_{\|f\|_\mu\leq1}\sup_{|\varphi(v)|\leq N}\frac{\mu(v)}{\mu(\varphi(v))}\mu(\varphi(v))\left|(C_\varphi(I-A_n)f)(v)\right| \\
	&=\sup_{\|f\|_\mu\leq1}\sup_{|\varphi(v)|\leq N}\frac{\mu(v)}{\mu(\varphi(v))}\mu(\varphi(v))\left|((I-A_n)f)(\varphi(v))\right| \\
	&\leq \sup_{\|f\|_\mu\leq1}\sup_{|\varphi(v)|\leq N}\frac{\mu(v)}{\mu(\varphi(v))}\sup_{|w|\leq N}\mu(w)\left|((I-A_n)f)(w)\right|. \\
	\end{aligned}$$  If $|w|\leq N$ and $n>N$, then $((I-A_n)f)(w)=0$ and we have $S_N(n)=0$.  Thus, for $n> N$, $$\|C_\varphi\|_e\leq \max\{R_N(n),S_N(n)\}\leq R_N(n)\leq\sup_{|\varphi(v)|\geq N}\frac{\mu(v)}{\mu(\varphi(v))}.$$ This estimate holds for all $N\in \N$, and hence $$\|C_\varphi\|_e\leq \lim_{N\rightarrow\infty} \sup_{|\varphi(v)|\geq N}\frac{\mu(v)}{\mu(\varphi(v))}.$$
	
	Now assume the essential norm of $C_\varphi$ is strictly less than the limit in equation \eqref{Equation:essentialnorm}.  Then there is a compact operator $K$ and constant $s>0$ such that $$\|C_\varphi-K\|<s<\lim_{N\rightarrow\infty} \sup_{|\varphi(v)|\geq N}\frac{\mu(v)}{\mu(\varphi(v))}.$$  Moreover, we can find a sequence of vertices $(v_n)$ with $|\varphi(v_n)|\rightarrow\infty$ such that \begin{equation}\label{equation:essentialnorm}\limsup_{n\rightarrow\infty}\frac{\mu(v_n)}{\mu(\varphi(v_n))}>s.\end{equation}  Now, define the sequence of functions $(f_n)$ by $$f_n(v)=\frac{1}{\mu(v)}\bigchi_{\varphi(v_{n})}(v).$$  By Lemmas \ref{Lemma:chi function} and \ref{lemma:bounded family}, this is a sequence of unit vectors in $\Lmuinf$ converging to zero pointwise.  We also have the lower estimate, $$s>\|C_\varphi-K\|\geq \|(C_\varphi-K)f_n\|_\mu\geq \|C_\varphi f_n\|_\mu-\|Kf_n\|_\mu.$$ By Lemma \ref{Lemma:Compactness-criterion}, $\|Kf_n\|_\mu\rightarrow 0$ as $n \to \infty$, and thus $$\begin{aligned}s &\geq\limsup_{n\rightarrow \infty}\left(\|C_\varphi f_n\|_\mu-\|Kf_n\|_\mu\right)\\
	&=\limsup_{n\rightarrow \infty}\|C_\varphi f_n\|_\mu\\
	&\geq\limsup_{n\rightarrow \infty} \mu(v_n)|f_n(\varphi(v_n))|\\
	&=\limsup_{n\rightarrow \infty} \frac{\mu(v_n)}{\mu(\varphi(v_n))}\\
	&>s,
	\end{aligned}$$ which is a contradiction.

    Next, suppose $C_\varphi$ is compact on $\Lmuinf$. If $\varphi$ has finite range, then we are done. If $\varphi$ has infinite range, then equation \eqref{limit condition} holds since the essental norm of a compact operator is zero.  

    For the converse, first suppose $\varphi$ has infinite range and equation \eqref{limit condition} holds.  Then $C_\varphi$ is compact by equation \eqref{Equation:essentialnorm}. Finally, suppose $\varphi$ has finite range, i.e. there exists positive constant $M$ such that $|\varphi(v)| \leq M$ for all $v \in T$.  Let $(f_n)$ be a bounded sequence in $\Lmuinf$ converging to 0 pointwise and fix $\varepsilon > 0$.  Since $\varphi(T)$ is finite, there exists a positive constant $m$ such that $\displaystyle\sup_{w \in \varphi(T)} \mu(w)  \leq m$.  Also, the pointwise convergence of $(f_n)$ to 0 is uniform on $\varphi(T)$.  Thus, for sufficiently large $n$, we have $\displaystyle\sup_{w \in \varphi(T)} |f_n(w)| < \frac{\varepsilon}{m \sigma_\psi}$.  With these observations, we see for such $n$,
    $$\begin{aligned}
    \|C_\varphi f_n\|_\mu &= \sup_{v \in T} \mu(v)|f_n(\varphi(v_n))|\\
    &= \sup_{v \in T} \frac{\mu(v)}{\mu(\varphi(v))} \mu(\varphi(v)) |f_n(\varphi(n)))|\\
    &\leq \sigma_\varphi \sup_{w \in \varphi(T)} \mu(\varphi(v)) |f_n(\varphi(n)))|\\
    &\leq m\sigma_\varphi \sup_{w \in \varphi(T)} |f_n(\varphi(n))|\\
    &< \varepsilon.
    \end{aligned}$$  So $\|C_\varphi f_n\| \to 0$ as $n \to \infty$.  Thus by Lemma \ref{Lemma:Compactness-criterion} $C_\varphi$ is compact on $\Lmuinf$.
\end{proof}

\begin{corollary} The bounded composition operator $C_\varphi$ on $\Linf$ has essential norm 0 or 1.  Furthermore, $C_\varphi$ is compact on $\Linf$ if and only if $\varphi$ has finite range.
\end{corollary}

One would expect that the characterization for compactness would be the ``little-oh" condition $\displaystyle\lim_{|v| \to \infty} \frac{\mu(v)}{\mu(\varphi(v))} = 0$ derived from the ``big-oh" condition of boundedness.  However, the following example shows this not to be true. 

\begin{example} Suppose the weight is defined as 
	$$\mu(v) = \begin{cases} 1 & \text{if } v=o,\\
	|v| & \text{if } |v| = 2n \text{ for some } n \geq 1,\\
	1 & \text{if } |v| = 2n-1 \text{ for some } n \geq 1.
	\end{cases}$$ As a matter of notation, for $n \in \N$ the vertex $w^n$ denotes a vertex of length $n$.  Fix a vertex $v_0$ such that $|v_0| = 1$ and define $\varphi:T\to T$ as
	$$\varphi(v) = \begin{cases}
	o & \text{if } v = o,\\
	w^{|v|^2} & \text{if } |v| = 2n \text{ for some } n \geq 1,\\
	v_0 & \text{if } |v| = 2n-1 \text{ for some } n \geq 1.
	\end{cases}$$
	
Observe that for vertex $v \in T$ with $|v| = 2n$ for some $n \geq 1$, we have $\displaystyle\frac{\mu(v)}{\mu(\varphi(v))} = \frac{1}{|v|}$.  For $v \in T$ with $|v| = 2n-1$ for some $n \geq 1$ we have $\displaystyle\frac{\mu(v)}{\mu(\varphi(v))} = 1$. Thus $$\lim_{N \to \infty} \sup_{|\varphi(v)| \geq N} \frac{\mu(v)}{\mu(\varphi(v))} = 0,$$ since $|\varphi(v)| \geq N$ implies $|v| = 2n$ for some $n \geq 1$.  Thus by Theorem \ref{Theorem:compactness}, $C_\varphi$ is compact on $\Lmuinf$.  However, $$\lim_{|v|\to\infty} \frac{\mu(v)}{\mu(\varphi(v))} \neq 0,$$ as can been seen by taking a sequence of vertices $(v_n)$ with $|v_n| = 2n-1$ for all $n \in \N$. 
\end{example}

\section{Isometries}\label{Section:Isometries}
In this section we investigate the isometric composition operators on $\Lmuinf$.  We provide sufficient conditions as well as necessary conditions for $C_\varphi$ to be an isometry.  In the case of $\Linf$, we obtain a characterization of the isometric composition operators.

\begin{theorem}\label{Theorem:sufficient} Let $\varphi$ be a self-map of a tree $T$ such that $C_\varphi$ is bounded on $\Lmuinf$.  If $\varphi$ is surjective and $\displaystyle\frac{\mu(v)}{\mu(\varphi(v))} = 1$ for all $v \in T$, then $C_\varphi$ is an isometry.
\end{theorem}
	
\begin{proof} Let $f \in \Lmuinf$.  Then 
$$\begin{aligned}
\|C_\varphi f\|_\mu &= \sup_{v \in T} \mu(v)|f(\varphi(v))|\\
&= \sup_{v \in T} \frac{\mu(v)}{\mu(\varphi(v))}\mu(\varphi(v))|f(\varphi(v))|\\
&= \sup_{v \in T} \mu(\varphi(v))|f(\varphi(v))|\\
&= \sup_{w \in T} \mu(w)|f(w)|\\
&= \|f\|_\mu.
\end{aligned}$$  Therefore $C_\varphi$ is an isometry on $\Lmuinf$.
\end{proof}

\begin{theorem}\label{Theorem:necessary} Let $\varphi$ be a self-map of a tree $T$.  If $C_\varphi$ is an isometry on $\Lmuinf$, then $\varphi$ is surjective and $\displaystyle\sup_{v \in T} \frac{\mu(v)}{\mu(\varphi(v))} = 1$.
\end{theorem}

\begin{proof}
Suppose $C_\varphi$ is an isometry on $\Lmuinf$.  Thus $\|C_\varphi\| = 1$.  By Theorem \ref{theorem:boundedness}, we have $\displaystyle\sup_{v \in T} \displaystyle\frac{\mu(v)}{\mu(\varphi(v))} = 1$. 
	
Assume $\varphi$ is not surjective.  Thus, there exists vertex $w \not\in \varphi(T)$.  Define the function $f(v) = \displaystyle\frac{1}{\mu(v)}\bigchi_w(v)$.  By Lemma \ref{Lemma:chi function}, $f \in \Lmuinf$ with $\|f\|_\mu = 1$.  Since $C_\varphi$ is an isometry, then $\|C_\varphi f\|_\mu = \|f\|_\mu = 1$.  However, we arrive at a contradiction since
$$\begin{aligned}
\|C_\varphi f\|_\mu &= \sup_{v \in T} \mu(v)|f(\varphi(v))|\\
&= \sup_{v \in T} \frac{\mu(v)}{\mu(\varphi(v))}\chi_w(\varphi(v))\\
&\leq \sup_{v \in T} \bigchi_w(\varphi(v))\\ 
&= 0,\end{aligned}$$ and thus $C_\varphi$ is not an isometry.  So $\varphi$ is surjective.
\end{proof}

The following example shows that the necessary conditions of Theorem \ref{Theorem:necessary} are not sufficient to induce an isometric composition operator on $\Lmuinf$.

\begin{example} Suppose $\mu(v) = |v|$ and $\mu(o)=1$ and define $$\varphi(v) = \begin{cases} v^- & \text{if } v \in T^*,\\o & \text{if } v = o.\end{cases}$$  Then $\varphi$ is surjective, as the tree has no terminal vertices.  Also, we have
$$\sup_{v \in T} \frac{\mu(v)}{\mu(\varphi(v))} = \sup_{v \in T} \frac{|v|}{|v|-1} = 1.$$  However, $C_\varphi$ is not an isometry on $\Lmuinf$.  Fix a vertex $w \in T^*$, and consider $\bigchi_w$.  On the one hand
$$\|\bigchi_w\|_\mu = \sup_{v \in T} \mu(v)\bigchi_w(v) = |w|$$ where as 
$$\|C_\varphi \bigchi_w\|_\mu = \sup_{v \in T} \mu(v)\bigchi_w(\varphi(v)) = |w|+1.$$
\end{example}

Although our conditions are not both necessary and sufficient for $C_\varphi$ being an isometry on $\Lmuinf$, we are able to conclude necessary and sufficient conditions for $C_\varphi$ to be an isometry on $\Linf$.

\begin{corollary}\label{Corollary:isometry} The bounded composition operator $C_\varphi$ is an isometry on $\Linf$ if and only if $\varphi$ is surjective.
\end{corollary}

If we add the assumption that $\varphi$ is injective, then we obtain necessary and sufficient conditions for $C_\varphi$ to be an isometry on $\Lmuinf$.

\begin{theorem}\label{Theorem:injective}
Let $\varphi$ be an injective self-map of a tree $T$.  Then $C_\varphi$ is an isometry on $\Lmuinf$ if and only if $\varphi$ is surjective and $\displaystyle\frac{\mu(v)}{\mu(\varphi(v))} = 1$ for all $v \in T$.
\end{theorem}

\begin{proof}
From Theorems \ref{Theorem:sufficient} and \ref{Theorem:necessary}, it suffices to show that if $C_\varphi$ is an isometry, then $\displaystyle\frac{\mu(v)}{\mu(\varphi(v))} = 1$ for all $v \in T$.  Suppose $\varphi$ is an injective map that induces an isometry $C_\varphi$ on $\Lmuinf$.  Additionally, assume there exists a vertex $x \in T$ such that $\displaystyle\frac{\mu(x)}{\mu(\varphi(x))} < 1$.  The function $f(v) = \displaystyle\frac{\mu(x)}{\mu(v)}\bigchi_{\varphi(x)}(v)$ is an element of $\Lmuinf$ with $\|f\|_\mu = \mu(x)$.  However, 
$$\begin{aligned}
\|C_\varphi f\|_\mu &= \sup_{v \in T} \mu(v) |f(\varphi(v))|\\
&= \sup_{v \in T} \frac{\mu(v)\mu(x)}{\mu(\varphi(v))} \bigchi_{\varphi(x)}(\varphi(v))\\
&= \frac{\mu(x)^2}{\mu(\varphi(x))}\\
&< \mu(x),
\end{aligned}$$ which contradicts $C_\varphi$ being an isometry.  Thus $\displaystyle\frac{\mu(v)}{\mu(\varphi(v))} = 1$ for all $v \in T$, as desired.  
\end{proof}

Lastly, we provide necessary conditions for $C_\varphi$ to be a surjective isometry.  The appeal of this result is the interplay between the point-evaulation functions and the adjoint of $C_\varphi$.

\begin{theorem} If $C_\varphi$ is a surjective isometry, then $\varphi$ is bijective and $\displaystyle\frac{\mu(v)}{\mu(\varphi(v))} = 1$ for all $v \in T$.
\end{theorem}

\begin{proof}
By Theorem \ref{Theorem:injective}, it suffices to show that if $C_\varphi$ is a surjective isometry, then $\varphi$ is injective.  The set of point-evaluation functions is invariant under the action of $C_\varphi^*$ (see \cite[Theorem 1.4]{CowenMacCluer:1995}).  Moreover, for every $v \in T$ we have $C_\varphi^* K_v = K_{\varphi(v)}$.  In addition, since $C_\varphi$ is a surjective isometry, then $C_\varphi^*$ is an isometry.

Suppose $v$ and $w$ are vertices in a tree $T$ such that $\varphi(v) = \varphi(w)$.  Then $$C_\varphi^* K_v = K_{\varphi(v)} = K_{\varphi(w)} = C_\varphi^* K_w.$$ Since $C_\varphi^*$ is injective, we obtain $K_v = K_w$.  By Lemma \ref{Lemma:equal functionals}, we have $v = w$, and thus $\varphi$ is injective.
\end{proof}

We conclude this section with an example of a symbol $\varphi$ that induces an isometric composition operator on $\Linf$ but not on all the $\Lmuinf$.

\begin{example}
Define the map $\varphi$ on a tree $T$ by $$\varphi(v) = \begin{cases} v^- & \text{if } v \in T^*,\\o & v = o.\end{cases}$$
On the space $\Linf$, $\varphi$ induces an isometric composition operator by Corollary \ref{Corollary:isometry}.  Interestingly, this map does not induce an isometric composition operator on other $\Lmuinf$ spaces.  Consider the weight $\mu(v) = 2^{|v|}$.  So $$\frac{\mu(v)}{\mu(\varphi(v))} = \begin{cases}2 & \text{if } v \neq o,\\1 & v = o.\end{cases}$$  Thus by Theorem \ref{Theorem:necessary}, $C_\varphi$ is not an isometry on $\Lmuinf$.
\end{example}

\bibliographystyle{amsplain}
\bibliography{references.bib}
\end{document}